\documentclass[11pt]{amsart}
\usepackage{graphicx,color}
\newtheorem{theorem}{Theorem}[section]
\newtheorem{lemma}[theorem]{Lemma}
\newtheorem{proposition}[theorem]{Proposition}
\newtheorem*{proposition*}{Proposition}
\newtheorem*{theorem*}{Theorem}
\newtheorem*{lemma*}{Lemma}
\theoremstyle{definition}

\theoremstyle{remark}
\newtheorem{remark}[theorem]{Remark}
\usepackage[english]{babel}
\usepackage[T1]{fontenc}
\usepackage{ifthen}
\usepackage{mathrsfs}
\usepackage{psfrag}
\numberwithin{equation}{section}

\DeclareMathOperator{\im}{Im}
\DeclareMathOperator{\re}{Re}

\begin{document}

\title{Subprincipal Controlled Quasimodes and Spectral Instability, II}

\author{Pelle Brooke Borgeke}
\address{Linn\oe us university}
\curraddr{}
\email{pelle.borgeke@lnu.se}

\subjclass[2010]{Primary}


\date {12 January 2026}
\begin{abstract} In this paper, we continue the analysis of the effects of \emph{subprincipal controlled} quasimodes, which are approximate solutions $u(h,b)$ to $P(h)u(h,b)\sim 0,$ depending on the subprincipal symbol $b$. These modes can give spectral instability (pseudospectrum) for the operator $P(x,hD_x;h^n B_{n\geq 1}),$ which has double zeros for the principal symbol. This means that $P(x,\xi)=p=dp=0$ in the neighborhood $\Omega \ni \nu=(x_0,\xi_0).$  In the first paper in this series, we considered operators with \emph{transversal} intersections of bicharacteristics. Now we study operators with \emph{tangential} intersections of bicharacteristics, as well as with double characteristics for $p$. We put the pseudodifferential operator on \emph{normal form} microlocally, and use the model operator $P(h)=hD_1(hD_1+\mathcal{Q}(x,\xi))+B(x,hD_x)$, with $\mathcal{Q}$ as a quadratic form in $(\xi,\underline{\xi}')=(\tau,\underline{\xi})$ to test for quasimodes. We demonstrate two cases where this happens, the $\beta$-condition and the $\partial_{\xi}{\beta}$-condition. We shall also continue with more advanced cases, when the operators are factorable to $P(h)=P_2(h)P_1(h,B)$, thus annihilating the subprincipal control over the quasimodes.
\end{abstract}
\maketitle
\section{Introduction}
This article, the second in a series of two, investigates the existence of semiclassical quasimodes, smooth functions $u(h,b)$ depending on the subprincipal symbol, $b$, such that $P(h)u(h,b)\sim 0$. These modes give spectral instability or pseudospectrum for \begin{equation}P(x,hD_x;h^n B_{n\geq 1}(x,hD_x)) = P(x,hD_x) +hB_1(x,hD_x)+h^2B_2(x,hD_x)+\ldots \end{equation} a pseudodifferential operator on asymptotic form, and here $hD_x=-hi\partial_x, P$ the principal part, $B_1$ subprincipal part, $B_{j\ge 2}$ “higher order terms” in the semiclassical parameter, $0<h \leq 1$. We can microlocaly factorize the principal part in (1.1) to $P(x,hD_x)= P_2(h)P_1(h)$ in the neigborhood $\Omega \ni \nu=(x_0, \xi_0)$ where we have $p_{\nu}=dp_{\nu}=0$ for the principal symbol, $p(x,\xi)$. The important phase space symbols that we study are \begin{align} (\sigma(P(x,hD_x))=p(x,\xi)) \  \land  \ (\sigma(B_1(x,hD_x))=b(x,\xi)=\alpha(x,\xi)+ i \beta(x,\xi)).\end{align}
We shall see that the latter, the subprincipal symbol, will be in control in $\Omega$, because of the double multiplicity for $p$, whether or not we find quasimodes. We use our microlocal \emph{model operator}  $P(h)=hD_1(hD_1+\mathcal{Q}(\xi))+B(x,hD_x), \ B \in \Psi^1,$ where $\mathcal{Q}(\xi)$ is a positive non-degenerate quadratic form and we have coordinates $(\xi_1,\xi_{j>1})=(\tau,\xi))$. We get pseudospectrum \begin{equation} \Vert R(h,0)\Vert= \Vert R(h)\Vert=  \frac{\Vert u(h,b) \Vert}{\Vert P(h)u(h,b) \Vert} \longrightarrow \infty, \ h\rightarrow 0. \end{equation} The resolvent $R(h,z)$ reduces to $R(h,0)=R(h)$ when we study $(P(h)+\zeta)$ as an operator; adding a constant does not change the results by more than a translation. Then $\zeta$ can be ignored (subtracted), we get $P(h)u(h,b)$, and we can thus use methods from the PDE theory, naturally modified for our needs. \footnote {We quote from page 2, in Zworski [12]: \emph{Some techniques developed for pure PDE questions, such as local solvability, have acquired a new life when translated to the semiclassical setting. An example is the study of pseudospectra of non-self-adjoint operators.}} We say that the operator has a $ \beta$-condition if there is a sign change in $ \beta$ or in $ \partial_{\xi}\beta$ along a limit bicharacteristic in $\Omega$. This means that the exponential function in $u(h,b)$ becomes a wave packet quasimode so that we get $(P(h)u(h,b) \sim 0) \land (h \rightarrow 0.)$
In the first article, we explored operators with transversal intersections of bicharacteristics; here, we study the \emph{tangential} type. We find that we locally, by using Taylor´s formula (and here $[\cdot]$ is for \emph{there exists}) \begin{align} [P(h)][\Omega](\sigma (P(h))=\xi_1(\xi_1 - q(x, \xi)), \sigma_{\text{sub}}(P(h))= b(x,\xi)\in S^1,\\ \re\sigma_{\text{sub}}(P(h))= \alpha(x, \xi), \im \sigma_{\text{sub}}(P(h))= \beta(x,\xi)  ).\end{align} We first show two types of $beta-$conditions, which give quasimodes:
1. The $\emph{imaginary part}$ $\beta(x, \xi)$ of the subprincipal \emph{symbol}, $b =\alpha+ i\beta$, must $\emph{change sign}$ on a $\emph{limit bicharacteristic}$, and we call this the $\beta$-condition.
2. The condition shifts so that the $\emph{derivative}$ of the imaginary part, $\partial_{\xi} \beta(x, \xi),$ must change sign and this is the $\partial_{\xi} \beta-$condition.
3. We also consider tangential bicharacteristics of higher order \begin{align}P(h)=hD_1(hD_1 - h^kD_2^k)+h^{j+1}B(x,D_x)D_2^j, k > 2, \\  h^kD_2^k \sim \mathcal{Q}(x, \xi_2^k) \end{align} and $\sigma_{\text{sub}}(P(h))=b(x,\xi) \xi_2^j$ and $ j\geq k$ thus generalizing the results from earlier findings. We examine the limitations of the $\beta$-conditions due to factorization depending on the numbers $j$ and $k$ so that we, in fact, annihilate the $\beta$-condition when $k=j,$ writing $P(h)=P_2(h)P_1(h,B)$. The material is divided as follows: We begin with a brief review of the proof method and the results from the first article. In section 3, we comment on some points in the proof method that were not addressed in the first article. In Section 4, we prove the Theorem with tangential intersection, and in Section 5, we work with the  $\partial_{\xi} \beta-$condition and the Factorization Theorem. In Section 6, we summarize and discuss this series of articles using a tableau that shows the different cases and conditions.
\section{Short review from the first article, and formulas for the tangential case}
We shall first recall some facts from the article [2]. (For the notation we use $(t,x)=(x_1,x_{j>1}) \land (\tau,\xi)=(\xi_1,\xi_{j>1}).)$ There we considered the transversal case and the semiclassical operator $$P(x,hD_x;h^n B_{n\geq 0}(x,hD_x)) = P(x,hD_x) + hB_1(x,hD_x)+h^2B_2(x,hD_x) \ldots$$ with the modell operator $P(h)=h^2D_1D_2+hB(x,hD_x), B \in \Psi^1.$ We worked with principal symbols $p$ that microlocally factor as $ p = p_1 p_2$ in the small neighborhood $\Omega$ we studied.
We recall the definitions we made before, and the sets we now work with are $$ \Sigma_2(P(h)) = { d_{\xi}p(x_0, \xi_0) = 0 }$$ and $$\Sigma_2(z) =  p^{-1}(z) \cap \Sigma_2(P(h)) = {x_0, \xi_0): (p(x_0, \xi_0)= z) \land (d_{\xi}p(x_0, \xi_0) = 0) }.$$
We write, and here $cl$ (classical) means a sum of terms, the subprincipal symbol as \begin{equation} \sigma_{\text{sub}}(P(h))|_{\Omega} = b(x,\xi) \in S_{cl}^1{(\Omega)} \end{equation} and the imaginary part \begin{equation} \im \sigma_{\text{sub}}(P(h))|_{\Omega} = \beta(x,\xi).\end{equation} The $\emph{limit bicharacteristics}$ ($\Gamma_j$) are the possible limits of bicharacteristics at $\Sigma_2(P(h))$ as we commented on in [2], where we proved the theorem \begin{theorem*} Let $P(x,hD_x;h^n B_{n\geq 0}(x,hD_x))$ have a real principal symbol $p(x,\xi)$ that microlocally factorizes $p=p_1p_2$ in the neighborhood $\Omega$ of $(x_0, \xi_0)$. Assume that $p^{-1}(z)$ is a union of two hypersurfaces with transversal involutive intersection at $\Sigma_2(P(h))$ and that $ d^2_{\xi} p(x_0, \xi_0) \not=0$ for $\zeta \in \Sigma (P(h))$. If the imaginary part of the subprincipal symbol $\beta(x,\xi)$ changes sign on a limit bicharacteristic in $\Sigma_2(z) \cap \Omega$, then $\zeta \in \sigma^{\infty}_{scips}(P(h))$, the semiclassical injectivity pseudospectrum of infinite order. \end{theorem*}
 The microlocal environment allow us to employ a normal form operator, instead of (1.1) in $\Omega$, where the approximate solutions are supported. We may thus use a simplified claim, via a proposition, to prove the theorem. In the first article, we introduced the parameters $\alpha$, $\beta$, and $\gamma$ in a system of transport equations. These are positive numbers in the interval (0,1) associated with the exponent of the semiclassical variable $h$ to adjust the system of equations in different ways. We change coordinates with $\gamma$, except for the $t$ direction where the change of sign shall occur, to get the subprincipal symbol on the form $b(t,h^{\gamma}x,h^{\beta}\xi_2)\sim b(t,h^{\beta}\xi_2)+\mathcal{O}(h^\gamma)$ (by Taylor)  in $\Sigma_2(z)$. This means that we take the usual association for $j=1$ \begin{equation} S^n(\mathbb{R}^n)\ni(x_j,\xi_j) \mapsto (x_j, hD_j)\in \Psi^n.\end{equation}
We make the change of coordinate in the operator space $\Psi^n$ by multiplying with $h^{\gamma}$ on $(x_j, \xi_j), j>1$ \begin{equation} \Psi^n(\mathbb{R}^n)\ni h^{\gamma}(x_j, hD_j)=(h^{\gamma}x_j, h^{1-\gamma}D_j)\in \Psi^n(\mathbb{R}^n).\end{equation} We used the following WKB-form \begin{equation} v_h(x)=e^{ig/h^{\alpha}}a(x) \end{equation} where $g(x) = x_2 \xi_2 $ and $e^{ig/h^{\alpha}}$ is the oscillating factor.  We have $\beta$ in the asymptotic expansion $a(x) = \sum_{j = 0}^N h^{j \beta} a_j$ so we get \begin{equation}v_h(x)=e^{ig/h^{\alpha}}  \sum_{j = 0}^N a_j h^{j \beta}, \quad a_j \in C^{\infty}.\end{equation}
The reason for taking the phase function $g$ in this way is to get the subprincipal symbol of $P(x,hD_x)$ in the first transport equation in system of equations. For the following we already made the scaling and used that $\beta = 1- \alpha -\gamma$ \begin{multline}e^{-i g/h^{\alpha}} P(x,h^{\gamma}x,D_1,h^{-\gamma}D_j)e^{i g/h^{\alpha}} a = hb(t,h^{\gamma}x,h^{\beta} \xi_2)a \\ + h\sum_j \partial_{\xi_j} p(x, h^{\beta} \xi_2) D_j a  \\ + h^2 \sum_{ij} \partial_{\xi_i}\partial_{\xi_j}p(x,h^{\beta} \xi_2) D_iD_ja +\mathcal{O}(h^3). \end{multline}
In a lemma, we gave a simplified formula for the Duistermaat expansion in the transversal case  \begin{lemma*} \begin{align} e^{-i x_2\xi_2 /h^{1-3\beta}} (P(h)e^{i x_2 \xi_2/h^{1-3\beta}} a )(x) \\ = h b(t, h^{\beta}\xi_2 )a+h^{1+\beta} \xi_2 D_1 a + h^{2-2\beta} D_1D_2 a +\mathcal{O}(h^{(1+2\beta}). \end{align} \end{lemma*} 
We also recall Lemma 2.6 from the first article to be able to transform the ansatz $v_h$ to $u(h)$ with $||u(h)||=1$ \begin{lemma*} We assume in the upper estimate for \begin{equation}v_h(x,b)=e^{ig/h^{\alpha}} \sum_{j = 0}^N a_j(x,b) h^{j \beta}, \quad a_j \in C^{\infty}. \end{equation} that we have \begin{align}(\phi_j)(h)[v_h][a_j][\mathscr{B}(t)]\bigl( ||v_h || \le C,\ a_j(t,x) = \phi_j(t,x)e^{-i\mathscr{B}(t)/h^{\beta}};\\  \mathscr{B}(t)\in C^{\infty},\im \mathscr{B} \le 0, \phi_j(t,x) \in C_0^{\infty}(\mathbb{R} \times \mathbb{R}^n), 0<h \ll 1 \bigr ). \end{align} For the lower estimate, we have \begin{align}[v_h] (\phi_0), [\mathscr{B}(0)] [\beta,c][h](||v_h|| \ge ch^{(n\alpha+\beta)/2}; \\ \phi_0(0,0) \not=0,  \mathscr{B}(0) = 0, \\  \beta, c>0,  0<h \ll 1). \end{align} The estimates are uniform if we have uniform bounds on $\mathscr{B} \ \land \ (j)(\phi_j).$ \end{lemma*}
After this resume, we now turn to the theme of this article, where we also use an adapted form of the expansion in the tangential case. Now we must factor out $h^{1+2\beta}$, so we use the result from Lemma 2.5 in [2]
\begin{equation}\left\{  
\begin{array}{ll}
1.\ \alpha = 1- (j_{=1|2}+2)\beta  \\
2.\ \beta =  \beta \\
3.\ \gamma =  (j_{=1|2}+1)\beta \\
4.\ \alpha + \beta + \gamma \\
=1- (j_{=1|2}+2)\beta+\beta+(j_{=1|2}+1)\beta=1
\end{array}\right.
\end{equation}
and take $\gamma=3\beta$, and $\alpha=1-4\beta$. We get the following expression after expansion, using Taylor's formula for $b$, factoring out, choosing the value of the parameter $\gamma$, and then converting to $\beta$
\begin{align} h^{1+2\beta}(h^{2(1-\beta)}D_1(h^{-3\beta} D_2+h^{\beta-1}\xi_2)^2a  \\ +h^{-2\beta}b_h(t)a +h^{1-2\beta}D_1^2a + \mathcal{O}(h^{\beta}) +(\mathcal{O}(h^{2-(3\mu-2)\beta})).\end{align} The last remainder term comes from the expansion, and we use it for control and estimates of the system of equations.
Finally, we turn to a lemma for the reduction in this case, which is done by using Taylor’s formula twice. \begin{lemma} With the assumptions about tangential intersections of bicharacteristics as in the Theorem 3.1 (below), the principal symbol can be reduced to the normal form with $p(x, \xi)|_{\Omega} =  \tau(\tau- q(x, \xi))$, where $q(x, \xi))$ is a non-degenerate quadratic form in $\xi' $, in a small neighborhood $\Omega \ni (x_0, \xi_0) \in \Sigma_2(P(h))$. \end{lemma} \begin{proof} First we subtract $\zeta$. In $\Omega \ni (x_0, \xi_0)$ we write $p^{-1}(z) = S_1 \cup S_2$ where $S_{j=1,2}$ are smooth hyper surfaces. If we make a symplectic change of coordinates, we may assume $S_1 = {\xi_1 = 0 }$ and then we get $S_2 = {\xi_1 = q(x,\xi)}$ where $q(x_0,\xi_0)= |dq(x_0,\xi_0)| = 0$ and $d^2q(x_0,\xi_0) \not=0$. Since $\Sigma_2(P(h))$ is an involutive manifold by assumption, we may complete $\xi_1$ to a symplectic coordinate system so that $\Sigma_2(P(h)) = { \xi_1 = … = \xi_k = 0 }.$
If we use Taylor’s formula for $p$, we find that $p = a\xi_1$, where $a=0$ on $S_2$. 
Another application of Taylor (after a change of coordinates) gives $a(x, \xi) = q_0(x, \xi)(\xi_1- q(x, \xi))$ and $$p(x, \xi) = q_0(x, \xi) \xi_1(\xi_1-q(x, \xi)).$$
We must have $q_0(x, \xi) \not=0$, else $d^2p(x_0, \xi_0) = 0$. By assumption $q(x, \xi) = dq(x, \xi)=0$ on $\Sigma_2(P(h))$ and $d^2q(x, \xi) \not= 0$ at $(x_0, \xi_0).$ Now we use the semiclassical quantization, canonical transformations, or FIOs \begin{equation}P(h) = q_0(x,hD)hD_1(hD_1 - q(x, hD_x)) + hB(x,hD_x) \end{equation} in the small neighborhood $\Omega$ of $(x_0, \xi_0)$. If we multiply with $q_0^{-1}(x,hD)$, we can take the elliptic $q_0=1$.  For lower-order terms, we may assume $B(x,hD_x) \in \Psi^1$ as before. \end{proof} \begin{remark} Note that we are localizing in the neighborhood $\Omega$ of $(x_0, \xi_0)$, the contributions from outside are thus $(N) (\mathcal{O}(h^N), N\in \mathbb{N})$, by the calculus. As before the approximate solution $u(h,b) \in L^2_{\Omega}$, is supported in the neighborhood where the normal form $p(x, \xi)|_{\Omega}=\xi_1(\xi_1 - q(x, \xi))$ exists. The geometry in this case is that the surfaces are tangent of exactly second order at $(x_0, \xi_0)$. We find that $H_{\xi_1} = H_{\xi_1 - q}$ on $\Sigma_2(0),$ so the set is foliated by lines in the $x_1$ direction, different from the transversal case where $\Sigma_2(0)$ is foliated by leaves. \end{remark} \begin{remark} More often, the Malgrange preparation theorem is used for reduction, and our proof resembles this theorem by using $q(x, \xi) = dq(x, \xi)=0$ on $\Sigma_2(P(h)$, but the technique here is more straightforward. In the first article, we referred to [1] for the reduction in the transversal case. It was done there for $h=1$, but this is not a restriction, because it is always possible to rescale to this case; see [12]. \end{remark}
\section{Operators with tangential intersection of bicharacteristics, $\beta$-condition} We now study the sign change in $t \mapsto \beta(t)$, where $\beta=\im \sigma_{\text{sub}}(P(h))$, as before, is the imaginary part of subprincipal symbol at $\Sigma_2(z)$ and we construct the quasimodes in a similar way as in the transversal case. \begin{theorem} Consider the semiclassical pseudodifferential operator with asymptotic expansion as $P(x,hD_x;h^n B_{n\geq 1})$.  Let $\zeta \in \Sigma (P(h))$ and assume there is $(x_0,\xi_0) \in p^{-1}(\zeta)$ such that $(x_0,\xi_0) \in \Sigma_2(\zeta), d^2p(x_0, \xi_0) \not=0$ and that $p^{-1}(\zeta)$ is a union of two hypersurfaces with tangential involutive intersection of precisely second order at $\Sigma_2(P(h))$, which is a manifold in a neighborhood $\Omega$ of $(x_0, \xi_0)$. If the imaginary part of the subprincipal symbol, $\im \sigma_{\text{sub}}(P(h))= \beta(x, \xi)$, changes sign on a limit bicharacteristic in $\Sigma_2(\zeta) \cap \Omega$, then $\zeta \in \sigma^{\infty}_{scips}(P(h))$, the injectivity pseudospectrum of infinite order. \end{theorem} The theorem is proved inderectly with the proposition to follow.  
We start by taking $q(x, \xi)=-\xi_2^2$ to make the calculation easier. We will, in Proposition 3.3, substitute the quadratic form. \begin{proposition} Let the principal symbol $p$ be factored $p=p_1p_2$ in the neighborhood $\Omega$. If $\im \sigma_{\text{sub}}(P(h))\arrowvert_{\Sigma_2(P(h))} = \beta(x,\xi)$, changes sign on a limit bicharacteristic in $\Omega$, then \begin{equation} (N)[u(h)][h](||P(h) u(h)|| \le  C_N h^{N}, N \in \mathbb{N}, ||u(h)||_{L^2} =1, 0<h \ll 1). \end{equation} We have the following solution formula for the transport equations, the quasimodes \begin{equation} a = e^{-i \int_0^t b_h(t)dt/\xi_2^2  h^{2\beta}} \end{equation} where $ D_1 i\int_0^t b_h(t)dt/\xi_2  = b_h(t)/ \xi_2^2$. \end{proposition} \begin{proof}We take $g(x) = x_2 \xi_2$ as before (again the symbol is on a form so that the eikonal equation disappears). If we include expansion, scaling, converting to the parameter $\beta$, Taylor in the $x$ coordinates for the subprincipals symbol $b$, and factorization, we get the following expression \begin{align} h^{1+2\beta}(h^{2(1-\beta)}D_1(h^{-3\beta} D_2+h^{\beta-1}\xi_2)^2a  \\ +h^{-2 \beta}b_h(t)a +h^{1-2\beta}D_1^2a +\mathcal{O}(h^{\beta}). \end{align} When we expand the square, we note the symmetry, inspect the expression \begin{align} h^{2(1-\beta)}D_1(h^{-3\beta} D_2+h^{\beta-1}\xi_2)^2 a \\ = D_1h^{2(1-\beta)+(\beta-1)2}\xi_2^2 a + 2h^{ 1-4\beta} \xi_2D_1D_2a+h^{2(1-4\beta)} D_1D_2^2a\\ +h^{-2 \beta}b_h(t)a +h^{1-2\beta}D_1^2a +\mathcal{O} (h^{\beta}), \end{align} and decide the first transport equation \begin{equation} D_1\xi_2^2 a +h^{-2 \beta}b_h(t)a=0. \end{equation} We now have \begin{align} [\alpha, \beta, \gamma](\alpha + \beta + \gamma = 1; \ \alpha=1-4\beta, \gamma=3\beta, \beta=\beta; \\ \alpha, \beta, \gamma \in \mathbb{R}(0,1)). \end{align} with $\mathcal{O}_{\mathcal{T}}$ as the remainder from Taylor \begin{align} h^{1+2\beta} (h^{-2 \beta}b_h(t)a+ \xi_2^2D_1a + h^{1-2\beta}D_1^2a + 2h^{1-4 \beta} \xi_2 D_1D_2a+h^{2(1-4\beta)}\\ D_1D_2^2a + \mathcal{O}_{\mathcal{T}}(h^{\beta})). \end{align} For the terms above, we get with $D_1 a= -h^{-2 \beta}b_h(t)a/\xi_2^2$ \begin{equation} h^{1-2\beta}  D_1^2a = h^{1-2\beta}  (-h^{-2 \beta}b_h(t)(-h^{-2 \beta}b_h(t))a/ \xi_2^2)=\mathcal{O}(h^{1-6\beta}) \end{equation} \begin{equation} 2h^{1- 4\beta}  \xi_2D_1D_2a = h^{1-4\beta} \xi_2 D_2(h^{-2\beta}b_h(x_1)a/ \xi_2^2)=\mathcal{O}(h^{1-6\beta}) \end{equation} \begin{equation} h^{2(1-4\beta)}  D_1D_2^2a = h^{2(1-4\beta)}  D_2^2(-h^{-2 \beta}b_h(x_1) a/\xi_2^2 ) = \mathcal{O}(h^{2(1-5\beta}). \end{equation}
For the remainder term for our two classes of operators, transversal and tangential intersections of bicharacteristics, we have $j=1$ for the first case and $j=2$ for the second case with the notation $h^{\kappa\beta}\xi_2^\kappa h^{\lambda}D_1^{\lambda}h^{\mu(1-\gamma)}D_j^{\mu},$ \begin{align}(h^{1+j_{1|2}\beta})\mathcal{O}(h^{\lambda+ \mu -1 + (\kappa- (j_{1|2}+1)\mu -j_{1|2})\beta }), \end{align} so here we get \begin{align} (h^{1+2\beta})\mathcal{O}(h^{\lambda+ \mu -1 + (\kappa- 3\mu-2)\beta}). \end{align} If we check for $\xi_2D_1D_2a$, we get $$\mathcal{O}(h^{2-1 +(1-3-2)\beta})=\mathcal{O}(h^{1-4\beta}).$$ And for $D_1D_2^2a$ we get $\mathcal{O}(h^{2-8\beta})$. Also, the other two equations are correct, checking with the general ordo-term, and we see that (as in the transversal case) $$\xi_2^2D_1a = \mathcal{O}(1)$$ and we solve in units of $\beta$, and we obtain the following terms of sizes $h^{\beta}, h^{3\beta}$, and $h^{1-4 \beta}$. As in the transversal case, we get \begin{align}a(t)=e^{-i \int b_h(t)dt /\xi_2^2 h ^{2\beta}}, a_j(t,x) =\phi(x) e^{-i \int b_h(t)dt /\xi_2^2 h ^{2\beta}}, \\ \phi \in C_0^{\infty}(\mathbb{R} \times \mathbb{R}^n), \phi(0)=1. \end{align}
The condition is as before that $\beta_h(t) \arrowvert_{ \Sigma_2(P(h))}$ changes sign in an interval near $0$, so we get by Taylor \begin{equation}t \mapsto  \beta_h(t) \sim \beta(t)\arrowvert_{ \Sigma_2(P(h))} + \mathcal{O}(h^{\beta}). \end{equation} For $h$ small enough this means that $ \int b(t)dt$ first increases on the interval and then decreases, so it has a maximum in the interval. We may integrate from the maximum, which we assume is zero, at $t=0$ so that $ \int_0^t b(t)dt \le 0$. We observe that the functions $b(t)$ are uniformly bounded in $C^{\infty}.$ Now we can apply this to an asymptotic expansion, and it is, with $b_j$ \begin{equation}P(h)v_h \sim e^{ix_2\xi_2 /h^{1-4\beta}} h^{1+2\beta} \sum_{j \ge 0} b_j h^{j\beta}. \end{equation} where $b_j$ are like $c_j$ in the transversal case or \begin{align}b_j \sim \xi_2^2 D_1a_j + h^{-2\beta}b(t) a_j + S_j(h) \ \sim e^{-i \int b(t)dt /\xi_2^2 h ^{2\beta}} (\xi_2D_1 \phi_j(x) + R_j(h)) = 0. \end{align} The terms $S_j(h)$ all have the exponential, so we can factor it out. The $R_j(h)$ is the type of term that is left, and it is just derivations of  $\phi_k$ uniformly bounded depending only on $\phi_k$ for $k < j$, following the remainder formula. We get as before with\begin{equation}h^{1-3\beta} D_1D_2 a_j = h^{1-3\beta} D_2(-S_n(h) -(h^{-2\beta}b(t)a_j/ \xi_2^2)) = \mathcal{O}(h^{1- 5\beta}) \end{equation} These terms will be entered later in the expansion. Theorem 3.1 now follows in the same way as was done in the first article [2]. \end{proof} We shall now handle the general case when $q(x, \xi))$ is a non-degenerate quadratic form in $\xi $, in the small neighborhood $\Omega \ni (x_0, \xi_0) \in \Sigma_2(P(h))$. \begin{proposition} Let the operator have the symbol \begin{equation} p = \tau(\tau + \sum_{jk}q_{jk} (t,x,\xi)\xi_j\xi_k) +hb(t,x,\tau,\xi). \end{equation} Then we can substitute $q(x, \xi))$ the non-degenerate quadratic form in $\xi $ to get \begin{align}h^{-2\beta}b_h(t)a+ q_0(t)D_1a + h^{1-2\beta}D_1^2a + 2h^{1- 4\beta} q_1(t,D_x) \\ D_1a + h^{2(1-4\beta)} q_2(t,D_x)D_1a = \mathcal{O}(h^{\beta}) \end{align} \end{proposition} \begin{proof} We get the scaled operator, and we use the standard expansion now \begin{multline} P(h) = h^2D^2_t + h^{3(1-2\beta)} \sum_{jk}q_{jk}(t,h^{3\beta}x,\tau, h^{-3\beta}\xi)D_tD_jD_k \\ + hB(t, x^{3\beta}, hD_t, h^{1-3\beta}D_x),\quad j,k>1 \end{multline} and if we take $g(x)= x \cdot \xi= x_2 \xi_2$  this will still give us the term $h^{2}D_t^2$. We will get the same equations as before, if we substitute $$\xi_2^2 \to q_0(t) = \sum_{jk}q_{jk} (t,0,0)\xi_j\xi_k $$ $$\xi_2D_2 \to q_1(t,D_x) = \sum_{jk}q_{jk}(t,0,0)\xi_jD_k $$ and $$D_2^2 \to q_2(t,D_x) = \sum_{jk}q_{jk}(t,0,0)D_jD_k. $$ The corresponding transport equation with the same values for the parameters is again \begin{align}h^{-2\beta}b_h(t)a+ q_0(t)D_1a + h^{1-2\beta}D_1^2a + 2h^{1- 4\beta} q_1(t,D_x) \\ D_1a + h^{2(1-4\beta)}q_2(t,D_x)D_1a = \mathcal{O}(h^{\beta}) \end{align} and we can proceed as when we used $ \xi_2^2$ as a stand in for $q(\tau, \xi)$ in the start of the proof.\end{proof} 
\section{Operators with ${\partial_{\xi}\beta}$-condition, Annihilation of ${\beta}$-condition in Theorem of Factorization} We now look into the case $\partial_{\xi}\beta(x, \xi)$-condition, where $\sigma_{\text{sub}}(P(h))$ becomes identically zero because there is a derivative connected to the subprincipal symbol. \begin{theorem} Let the conditions be as in Theorem 3.1, but now  $P(h)= hD_1(hD_1 + h^2D^2_2) + h^2B(x,hD_x)D_2, $ with $\sigma_{\text{sub}}(P(h))=b(x,\xi) \xi_2$. If the $\emph{derivative}$ of the imaginary part of the subprincipal symbol ${\partial_{\xi}\beta(x, \xi)}$ changes sign on a limit bicharacteristic in $\Sigma_2(z) \cap \Omega$, then $z \in \sigma^{\infty}_{scips}(P(h))$, the injectivity pseudospectrum of infinite order.
We have the following solution formula \begin{equation}a = e^{-i \int_0^t b_h(t)dt/\xi_2  h ^{\beta}} \end{equation} for the transport equations, the quasimodes, where $ D_1 i \int_0^t b_h(t)dt/\xi_2  = b_h(t) / \xi_2$ and $a_0=a$ and $a_j(t,x) =  \phi_j(x)e^{-i \int_{0}^t b_h(t)dt/\xi_2  h ^{\beta}}$ in the asymptotic expansion. \end{theorem} \begin{proof} The model operator is now $$P(h)= hD_1(hD_1+h^2D_2^2) + h^2 B(x,hD_x)D_2$$ and the expansion, scaling and factoring out gives the subprincipal symbol \begin{align} hb(t, h^{3\beta}x,h^{\beta} \xi_2)h^{\beta}\xi_2 a = h^{1+2\beta}(h^{-\beta}b(t, x^{3\beta}, h^{\beta} \xi_2) \xi_2 a. \end{align} The rest is the same as before \begin{align}h^{1+2 \beta}(\xi_2^2D_1a + h^{1- 2\beta}D_1^2a + 2h^{\alpha} \xi_2 D_1D_2a + h^{2\alpha} D_1D_2^2a \\  +h^{-\beta}b(t,h^{\beta}\xi_2)\xi_2 a + \mathcal{O}(h^{\beta})) \end{align} and we solve modulo terms that are $\mathcal{O}(h^{\beta})$ and the solution is \begin{equation}a(t,x) =\phi(x) e^{-i \int b_h(t)dt /\xi_2h ^{\beta}} \end{equation} as we only have $h^{-\beta}$  in front of $b$ and there is a $\xi_2$-factor in the exponent coming from $\sigma_{\text{sub}}(P(h))=b(x,\xi) \xi_2$. \end{proof} Now the subprincipal symbol is identically zero at $\Sigma_2(P(h))$, so when we look at the condition we get, and we remember that $\beta \in \mathbb{R}(0,1)$ in the exponent, \begin{equation} t \mapsto \beta(t,h^{\beta}\xi_2)= \beta(t)\arrowvert_{ \Sigma_2(P(h))}+h^{\beta} \partial_{\xi_2} \beta(t)\arrowvert_{\Sigma_2(P(h))}\xi_2+ \mathcal{O}_{\mathcal{T}}(h^{2\beta}). \end{equation} 
If $t \mapsto \beta(t)$ changes sign we have the first case, if $\beta(t) \equiv 0 $ then $t\mapsto h^{\beta} \partial_{\xi_2} \beta(t,0)\xi_2$ changes sign and we get quasimodes in the the $ \partial_{\xi_2}(\beta(x, \xi))$-condition. \footnote{We cannot also have $h^{\beta} \partial_{\xi_2} \beta(t)\arrowvert_{\Sigma_2(P(h))}\xi_2=0$, as this means that it is not possible to make an analysis on $b$, the same that happened when we abandoned the principal symbol. This is sometimes called an operator of \emph{subprincipal type}.}
So far, we have looked into the cases where we have $$ P(h)=hD_{x_1}(hD_{x_1}-Q(x,hD_{x'}) + h B(x,hD_x)$$ and $$P(h)= hD_1(hD_1 + h^2D^2_2) + h^2B(x,hD_x)D_2,$$ which we now write a little bit differently, because we like to generalize Theorem 4.1 in a natural way. We increase the exponent $k>0$ in $\xi_2^k$ that is linked to the subprincipal symbol, and we also do the same for $\xi_2^j$ using the positive numbers $j \land k$ with fixed and bounded value, in the exponent \begin{equation} P(h)= h^2D_1(hD_1-h^jQ(x,hD_{x'}) + h^k  hB(x,hD_x)h^kQ(x,hD_{x'}).\end{equation}
First, we recall the factorization result in the transversal case that annihilates the subprincipal control of the quasimodes, along with the proof, which can be useful.\begin{theorem*} For the transversal case, we have the model operator as \begin{align}P(h)=hD_1hD_2+hB(x,hDx)\\ =h^2D_1D_2+A_1(x)h^2D_1+A_2(x)h^2D_2+hR(x)) \end{align} In the cases when $R(x)=0$, we can factorize modulo $\Psi^0$ and get $h^2P_2P_1v_h=0$ and no quasimodes, and we have no double multiplicity for $p$ any longer. \end{theorem*} \begin{proof} We let $R(x)=0$ and for $(\text{mod} \Psi^0)$ \begin{align} P(h)=hD_1h^jD_2+A(x)h^kD_2 \cong h^{j+k}D_2(D_1+A(x)) = h^2P_2P_1. \end{align} Here, we do not get quasimodes because of factorization to $h^{j+k}P_1P_2$ and the solution $$a(t) =e^{-i \int b(t)dt)/(h^{n\beta}\xi_2^{n})}$$ becomes here instead $a(t)=e^{-i  \int b(t)dt}, \ n=k-j=0; \ j,k=1,2.$ When we consider the differential \begin{align}d_{\xi}p(x,\xi)=d_{\xi}(\xi_1(\xi_2+B(x))=(\xi_2+B(x))d\xi_1+\xi_1 d\xi_2) \not =0 \land (\xi_1=\xi_2=0). \end{align} Recall that for $dp(\xi)=dp(\xi_1\xi_2)=(\xi_2 d\xi_1 + \xi_1d\xi_2)=0) \land (\xi_1=\xi_2=0.)$ \end{proof} After this, we look at \begin{theorem} Consider the model operator for the tangential case in the form $$P(h)=hD_1(hD_1 + h^kD_2^k) + hB(x, hD_x)h^jD_2^j.$$ We put $Q(x,hD_x)=h^{j|k}D_2^{j|k}$ with symbol $q(x,\xi) = q(\xi^{j|k}_2).$ For $n>0$ we have $$\xi_2^kD_1a=h^{-n\beta}b(t,h^{\beta} \xi_2)\xi_2^j a.$$ The solution is $$a(t) =e^{-i \int b_h(t)dt)/(h^{n\beta}\xi_2^{n})}$$ and we get quasimodes as before, but when $0=n=k-j,(k=j)$ we instead get $$P(h)=P_1(h)P_2(h)$$ and no quasimodes in $\sigma_{scips}^{\infty}$ as $P_j(h), j=1,2$ now lacks the subprincipal control and $$a(t) =e^{-i  \int_0^t b(t)dt}.$$ \end{theorem} \begin{proof} For $n>0$ we see this by induction for $n \in \mathbb{N}$ and the cases $n=1,2$ are allready clear by $$a(t,x) =\phi(x) e^{-i \int b_h(t)dt /\xi_2^2h ^{2\beta}}$$ for $n=k-j=2-0=2$ and by $$a(t,x) =\phi(x) e^{-i \int b_h(t)dt /\xi_2h ^{\beta}}$$ where $k=2$ and $j=1$ gives $n=1$. For the second case, $j=k$ $$P(h)=hD_1(hD_1 + h^kQ) + hB(x, hD_x)h^kQ$$ we may factorize mod $\Psi^0$. We get with $P_1(h)=h(D_1+B), B\in \Psi^0$ \begin{align} P(h)=hD_1(hD_1 + h^kQ) + hBh^kQ= h^2D_1^2 + h(D_1+B)h^{k}Q \\=h^2D_1^2 + P_1(h)h^{k}Q. \end{align} For $P_1(h)$ the same and now $P_2(h)$
\begin{equation}\left\{  
\begin{array}{ll}
P_1(h)=h(D_1+B) \\
P_2(h)=P_1(h)+h^{k}Q-2hB
\end{array}\right.
\end{equation} we find \begin{align} (P(h)\cong P_1(h)^2+(h^{k}Q-2hB)P_1(h))\  \\ = ((P_1(h)+h^{k}Q-2hB)P_1(h)\cong P_2(h)P_1(h)) \land (\text{mod} \Psi^0). \end{align} as we may write $(\text{mod} \Psi^0)$  $$P_1(h)^2+(h^{k}Q-2hB)P_1(h))$$ $$=h^2(D_1^2+2BD_1+B^2)+P_1(h)h^{k}Q-2h^2BD_1-2h^2B^2 \cong h^2D_1^2 + P_1(h)h^{k}Q,$$ which we had above.
This procedure means we have factored out the dependence on lower-order terms, which again does not yield any quasimodes. We have 
\begin{equation}\left\{ 
\begin{array}{ll}
P_1(x,\xi)=\xi_1+B \\
P_2(x,\xi)=\xi_1-2B+\xi_2^k\\
P_2P_1(x,\xi)=\xi_1^2-B\xi_1+\xi_1\xi_2^k+B^2+B\xi_2^k
\end{array}\right.
\end{equation} For control of the differential with factorization we find \begin{align}d(P_2P_1(x,\xi))=(2\xi_1-B+\xi_2^k)d\xi_1+ (k\xi_1+kB)\xi_2^{k-1}d\xi_2 \not = 0 \\ \land (\xi_1=\xi_2=0), \end{align} which is the same as in the transversal case above, (4.11).
We can now also use a priori estimate from the first article(Appendix B) coming from the definition of semiclassical injectivity pseudospectrum \begin{equation} \Vert P(h)u(h,b) \Vert > Ch^{N}.\end{equation}
We write for the neighborhood $\Omega$ where $u \in C ^{\infty}_0$ has support$(\rho)$ \begin{align} (x)[\rho] [u] [C]\bigl(||u||\leq C \rho ||P_1(h)u|| \\ \leq C \rho^2 ||P_1(h)P_2(h)u|| \leq C \rho^2 \textbf{.}||P(h)u||+ ||u||;\\ x \land x_0 \in \Omega, |x \leq \rho \leq x_0|, ||u|| = 1, C=C_{0|1|2}\bigr).\end{align} We now have $u$ instead of $u(h,b)$ as none of ($P(h) | P(h)_{1\lor 2})$ can give quasimodes.\end{proof}
We summarize these findings in a final Factorization Theorem. \begin{theorem} Let the operator with transversal intersections of bicharacteristics be written $$ (i) \ P(h)=hD_1h^jD_2+A(x)h^kD_{1|2}\cong h^{j+k}D_{2|1}(D_{1|2}+A(x)) = h^2P_1P_2.$$ If this factorization, $mod \ \Psi^0$, is possible, there will be no quasimodes.
For the operator in the tangential case $$ (ii) \ P(h)=hD_1(hD_1 + h^jQ(t,x;\xi)) + h^{k+1}B(x, hD_x)Q(t,x,\xi)$$ with $(Q(t,x,\xi)\sim \xi_2^{j|k}) \land (j,k)\in \mathbb{N}$  of fix bounded order). Then, for $j>k$, the condition for quasimodes is a sign change for the derivative of the subprincipal symbol, but when $k=j$, it is possible to factorize to $P_2(h)P_1(h,B),$\ in modulo $\Psi^0;$ the control from the subprincipal symbol is lost, and we do not get quasimodes. \end{theorem} \begin{proof} We shall also show that $q(x,\xi)hD_1a $ is bounded as before. Recall that for the remainder we have, here with $\kappa=j$ $$(h^{1+j_{1|2}\beta})\mathcal{O}(h^{\lambda+ \mu -1 + (\kappa- (j_{1|2}+1)\mu-j_{1|2})\beta }).$$ In this case we get $$q(x,\xi)hD_1 \sim h^{1+j\beta}\xi^j_2D_1a = (h^{1+j\beta})\mathcal{O}(h^{\lambda+ \mu -1 + (j- 2\mu +j)\beta })=(h^{1+j\beta})\mathcal{O}(1).$$
For the Taylor expansion in $x'$ for $h^{1+k\beta}\xi_2^k b(t,x^{\gamma},h^{\beta}\xi_2))$ we get $$h^{1+j\beta}(h^{(k-j)\beta}(b(t,h^{\beta}\xi_2+\mathcal{O}(h^{\gamma-j\beta}))$$ taking $\gamma-j\beta=(j+1)\beta-j\beta=\beta.$\end{proof}
\section {Summation and Discussion of Results in this Series, Background and Bibliography}
We shall now summarize and discuss our findings and comment on some of our sources in the Bibliography. To help with the summation and discussion, we made a table below with different cases and conditions: $(c/c)=c_{(j_{1|2|3}/k_{1|2|3|4})}.$
\begin{center}
\begin{tabular}{ l| l | l | l| l |l |l |l|l |}
case/condition  & 1.$\beta$& & 2.$\partial_{\xi} \beta$&  &3.$P_2(h)P_1(h,B)$ & & 4.$\alpha$ & \\
\hline
1. Transversal & $c_{(1/1)}$&$\surd $ & $c_{(1/2)}$&  &  $c_{(1/3)}$ & $\surd $ & $c_{(1/4)}$& \\
\hline
2. Tangential & $c_{(2/1)}$& $\surd $ & $c_{(2/2)}$&$\surd $ &$c_{(2/3)}$& $\surd $ & $c_{(2/4)}$& \\
\hline
3. Factorable & $c_{(3/1)}$ &  &$c_{(3/2)}$&$\surd $ &  $c_{(3/3)}$ & $\surd $& $c_{(3/4)}$& \\
\hline
\end{tabular}
\end{center}
In this series of articles, we consider special semiclassical pseudodifferential operators that, microlocally, have a principal symbol, as a product $P(\xi)=P_1(\xi)P_2(\xi) $. We find double multiplicity for the principal symbol $p=p_2(\xi)p_1(\xi) \in S^k, k$ is of fix bounded order, for some point $\nu=(x_0,\xi_0)\in \Omega.$ This means that $p=dp=0$. Of course $$(p_1p_2=0) \land (p_1(\nu) \lor p_2(\nu)) =0 ; \ dp= (p_2dp_1+p_1dp_2=0) \land (p_1(\nu) \land p_2(\nu)=0).$$ Our model operator looks like \begin{align} P_1(h)P_2(h)+hB(h)=h^2D_1D_2+hB(h); \ hD_1(hD_1+h^2D_2^2)+hB(h) \\ h^2D_1D_2+hB(h)hD_2; \ hD_1(hD_1+\mathcal{Q}_j)+hB(h)\mathcal{Q}_k \\ \text{and here} \ \mathcal{Q}_{j|k}(\xi) \sim \xi_2^{j|k}.\end{align}
In (5.2), an extra derivative is attached to the subprincipal part, making it factorable when $k=j$.
In the PDE theory of solvability, the transversal case has been considered, especially in [3], [4], and [8]. There the focus is on the adjoint operator $P^*$  and an estimate of the type $||u ||\leq ||P^*u||$ that  garanties solvability (with all hypotheses fulfilled) by the \emph{non-existens} of quasimodes in  $\nu=(\xi_1=\xi_2=0).$  Outside $\Omega$ the operator has no double multiplicity and the subprincipal symbol $B_1(x,\xi)$ gives back control to the principal symbol. In the neighborhood $\Omega \ni (\xi_1=\xi_2=0)$ for the transversal case we have $(H_{p_1}=\partial_{x_1}) \land (H_{p_2}=\partial_{x_2})$ so the bicharacteristics are lines in the $x_1$ and $x_2$ direction foliating $\Omega$ with transversal intersections, in this way generating planes (leaves). For this case, the $\beta$-condition $(c/c)=(1/1)$ in the table could be seen as the semiclassical analog of the condition in [8], but there are differences. We study a sign change at $t=0$ along a \emph{limit} bicharacteristic for the imaginary part of the subprincipal symbol. The limit bicharacteristic that we use is to be understood as we approach from outside the foliation, because $\partial_\xi=0$, so the bicharacteristic is just a point. This notion is not used in [8], but in [1], page 44, from 2012, we write:
“We call the possible limits of bicharacteristics at $\Sigma_2(P(h))={dp_0(x_0,\xi_0)=0}$ the \emph{limit bicharacteristics}”
\footnote{The limit of bicharacteristics was defined rigorously in [4] 2016 by “We say that a sequence of smooth curves $\Gamma_j$ on a smooth manifold converges to a smooth limit curve $\Gamma$ (possibly a point) if there exist parametrizations on uniformly bounded intervals that converge in $C^{\infty}$. If $p \in  C^{\infty}(T^*X)$, then a smooth curve $\Gamma \subset \Sigma_2 S^*X$ is a limit bicharacteristic of $p$ if there exists bicharacteristics $\Gamma_j$ that converge to it”.}
Also, here we do not need the scaling as in [8], where $ s_1=0$, $ 0 < s_2 < s_3 = … =s_n < 1$, and $2s_2=s_n$ as we approach the point $t=0$. We instead obtain the subprincipal symbol on the form $b(t) + \mathcal{O}(h^{\beta})$ by Taylor-expanding twice and then letting the semiclassical parameter $h\rightarrow 0$. Our $\gamma$-scaling is used to work with Taylor in the $(t,\underline{x})$ coordinates and to balance the remainder, so it serves different purposes. The parameters $\alpha, \beta$ and $\gamma$ are instead balancing factors in the system of transport equations, where the crucial quotient is $(\frac{\gamma }{\beta})=j_{2|3}\beta,$ to get a limited remainder term, while balancing the system of equations. The proof method differs: we use a proposition to prove the theorem indirectly, not used in [8]. Moreover, we have simplified and generalised the method compared to [1], making it clearer and easier to use.
The $(c/c)=(1/1)$ is a good start, and there we also advanced our knowledge of the special method we began using in [1]. We have also gained new insights into how the subprincipal symbol can be composed to obtain quasimodes. It does not work with an extra derivative $c_{(1/2)}$ because now it is possible to factorize with $(j=k=2)$ $$h^jD_1D_2+ h^kD_{1|2} B(x,hD_x)=h^jD_{2|1}(D_{1|2}+ B(x,hD_x))$$ and we lose the subprincipal control. In [3], we find a definition of an operator of subprincipal type, demanding that $\partial_\xi(b) \not = 0$. So this terminology is the same as for the principal symbol in the theory of local solvability: principal type, not principal type. These conditions are more commonly known as Condition$(\Psi)$, Sub$(\Psi)$, etc, but here we prefer the more straightforward approach that points to the fact that the conditions are usually connected to the imaginary part of a symbol, using a common way to write imaginary functions and numbers as in $\zeta=\alpha+i\beta$,  It also has the benefit that we have a terminology for pseudospectra that is not related to the problem of solvability. The subjects are related, but they are not twins. Continuing in the table, we find the $\alpha$-condition $=(1/4)$. This can be the case when $\beta$ is constant on the leaves of foliation, so we cannot have a change of sign; instead, the real part $\alpha$ can play a role. We have not advanced far here, but we have identified that this case $\cdot/4$ does exist.  Eventually, this can mean that the theorem in [8] must be weakened by imposing conditions on how $B(x,D_x)$ must be composed for the theorem to hold. It will hold if one treats $B$ as just  $B=\re B +\im B$ but as the reduction gives $B(x,D_x)=A_1(x)D_1+A_2(x)D_2+ R(x)$ you may think that it should work for all combinations, but on our side we do not get quasimodes due to factorization when $R(x)=0$, so this can have an impact on the question of solvability in [8].
For the tangential case on the next line $(2/\cdot)$ in the table, in contrast with the transversal case, this has no prior counterpart in the theory of solvability. We find that in the neighborhood $\Omega \ni (\xi_1=\xi_2=0)$ we have  $H_{\tau}=H_{\tau-q}$ so the bicharacteristics are just lines in the $t$ direction foliating  $\Omega.$
We also studied this in [1], where we set $\alpha=5/12$, $ \beta=1/12$, and $\gamma=1/2$. This was a questionable choice, because $\frac{\gamma}{\beta}=j=6$ and we only need $j=3$, so this we shift to $\alpha=1-4\beta, \beta=\beta$ and $\gamma=3 \beta$. As you notice $1/2=\gamma>\alpha$ was corrected as otherwise we get remainders $\mathcal{O}_T(h^{1+\gamma})=\mathcal{O}_G(h^{2-\gamma})$ to get $\mathcal{O}_{T+G}$ and this we want to avoid. Also, we can get $\alpha-\gamma=-1/12$ for the values chosen in [1], and this term can appear in the system, and we cannot find factors $(\gamma-\alpha)$ there. The $c/c=2/2$ was demonstrated in [1], but here we have pointed out what this factor of derivation can add to the subprincipal symbol; we then studied its effect. The effect is two-fold: first, the condition shifts to the derivative of the subprincipal symbol when $k>j$; second, when $k=j$, it is possible to factorize, and then the effect from the subprincipal symbol is lost, so we do not get quasimodes.
For the Bibliography, we have mostly used Dencker in [3] and [4]. A great resource has also been Trefethen-Embree [10], and Zworski [12] for the theory of pseudospectrum and for results in the semiclassical field. And of course [6] and [7] are the firm references for all working with linear differential equations. For the notation and terminology, we were inspired by the classic work of Whitehead and Russel, [11].

\end{document}